\documentclass[11pt]{amsart}
\usepackage{mathrsfs}
\usepackage{times}
\usepackage[T1]{fontenc}
\usepackage{mathrsfs}
\usepackage{latexsym}
\usepackage[dvips]{graphics}
\usepackage{epsfig}
\usepackage{amsmath,amsfonts,amsthm,amssymb,amscd}
\input amssym.def
\input amssym.tex
\usepackage{color}
\usepackage[all]{xypic}

\newtheorem{theorem}{Theorem}[section]
\newtheorem{lemma}[theorem]{Lemma}

\newtheorem{figuretext}{Figure}
\theoremstyle{remark}
\newtheorem{remark}[theorem]{Remark}
\newtheorem*{remark*}{Remark}
\theoremstyle{remark}
\newcommand{\R}{{\Bbb R}}

\newcommand{\Z}{{\Bbb Z}}

\begin{document}

\title{Stability of the $\mu$-Camassa-Holm Peakons}
\author{Robin Ming Chen}
\address{Robin Ming Chen\newline
School of Mathematics\\
University of Minnesota\\
Minneapolis, MN 55455} \email{chenm@math.umn.edu}
\author{Jonatan Lenells}
\address{Jonatan Lenells\newline
Department of Mathematics\\
Baylor University\\
Waco, TX 76798-7328} \email{Jonatan\_Lenells@baylor.edu}
\author{Yue Liu}
\address{Yue Liu\newline
Department of Mathematics, University of Texas at Arlington, Arlington, TX 76019-0408}
\email{yliu@uta.edu}

\thanks{The work of R.M. Chen was partially supported by the NSF grant DMS-0908663. The work of Y. Liu was partially supported by the NSF grant DMS-0906099 and the NHARP grant 003599-0001-2009.}

\maketitle \numberwithin{equation}{section}

\begin{abstract}
The $\mu$-Camassa-Holm ($\mu$CH) equation is a nonlinear integrable partial differential equation closely related to the Camassa-Holm equation. We prove that the periodic peaked traveling wave solutions (peakons) of the $\mu$CH equation are orbitally stable.
\end{abstract}




\section{Introduction}
The nonlinear partial differential equation
\begin{equation}\label{muCH}
  \mu(u_t)-u_{xxt}=-2\mu(u)u_x+2u_xu_{xx}+uu_{xxx}, \qquad t > 0,  \quad x \in S^1 = \R/\Z, \\
\end{equation}
where $u(x,t)$ is a real-valued spatially periodic function and $\mu(u)=\int_{S^1}u(x,t)dx$ denotes its mean, was recently introduced in \cite{KLM} as an integrable equation arising in the study of the diffeomorphism group of the circle. It describes the propagation of self-interacting, weakly nonlinear orientation waves in a massive nematic liquid crystal under the influence of an external magnetic field.
The closest relatives of (\ref{muCH}) are the Camassa-Holm  equation \cite{C-H, F-F}
\begin{equation}\label{CH}
u_t-u_{txx}+3uu_x=2u_xu_{xx}+uu_{xxx},
\end{equation}
and the Hunter-Saxton \cite{H-S} equation
 \begin{equation}\label{HS}
-u_{txx}=2u_xu_{xx}+uu_{xxx}.
\end{equation}
In fact, each of the equations (\ref{muCH})-(\ref{HS}) can be written in the form
\begin{equation}\label{CHmform}
m_t + u m_x + 2 u_x m = 0, \qquad m = Au,
\end{equation}
where the operator $A$ is given by $A= \mu - \partial_x^2$ in the case of (\ref{muCH}), $A = 1 - \partial_x^2$ in the case of (\ref{CH}), and $A = - \partial_x^2$ in the case of (\ref{HS}).
Following \cite{lmt}, we will refer to equation (\ref{muCH}) as the $\mu$-Camassa-Holm ($\mu$CH) equation.

Equations (\ref{muCH})-(\ref{HS}) share many remarkable properties: (a) They are all completely integrable systems with a corresponding Lax pair formulation, a bi-Hamiltonian structure, and an infinite sequence of conservation laws, see \cite{C-H, C-M, H-Z, KLM}. (b) They all arise geometrically as equations for geodesic flow in the context of the diffeomorphism group of the circle $\text{Diff}(S^1)$ endowed with a right-invariant metric \cite{KLM, K-M, kou, mis2, Shk98}. (c) They are all models for wave breaking (each equation admits initially smooth solutions which break in finite time in such a way that the wave remains bounded while its slope becomes unbounded) cf. \cite{C-H, con1, con4, C-M, H-S, KLM, mis1}.

A particularly interesting feature of the Camassa-Holm equation is that it admits peaked soliton solutions \cite{C-H}. These solutions (called peakons) are traveling waves with a peak at their crest and they occur both in the periodic and in the non-periodic setting. It was noted in \cite{lmt} that the $\mu$CH equation also admits peakons: For any $c \in \R$, the peaked traveling-wave $u(x, t) = c\varphi(x - ct)$, where (see figure \ref{peakon.pdf})
\begin{equation}\label{varphipeakon}
  \varphi(x) = \frac{1}{26}(12 x^2 + 23) \quad \text{for} \quad x \in [-1/2, 1/2]
\end{equation}
and $\varphi$ is extended periodically to the real line, is a solution of (\ref{muCH}).
Note that the height of the peakon $c\varphi(x-ct)$ is proportional to its speed.

If waves such as the peakons are to be observable in nature, they need to be stable under small perturbations. The stability of the peakons is therefore of great interest. Since a small change in the height of a peakon yields another
one traveling at a different speed, the correct notion of stability here is that of {\it orbital stability}:
a periodic wave with an initial profile close to a peakon remains close to some translate of it for all later times. That is, the shape of the wave remains approximately the same for all times.

The Camassa-Holm peakons are orbitally stable in the non-periodic setting \cite{C-S} as well as in the periodic case \cite{le1}.
In this paper, we show that the periodic $\mu$CH peakons given by (\ref{varphipeakon}) are also orbitally stable:

\begin{theorem}\label{thm_stabintro}
The periodic peakons of equation (\ref{muCH}) are orbitally stable in $ H^1({S^1}). $
\end{theorem}
An outline of the proof of thereom \ref{thm_stabintro} is given in section \ref{outlinesec}, while a detailed proof is presented in section \ref{proofsec}.
We conclude the paper with section \ref{commentssec} where we discuss some results on the existence of solutions to (\ref{muCH}).

\begin{figure}
   \begin{center}
    \includegraphics{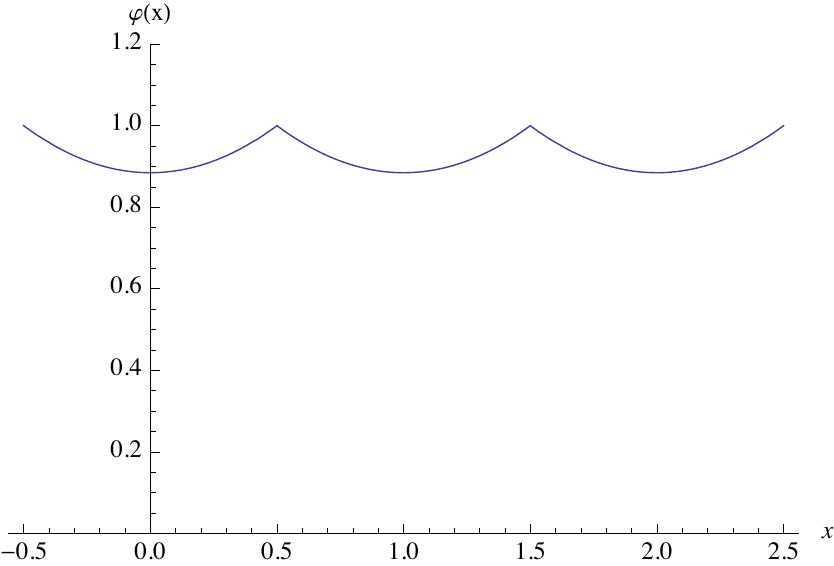} \\
    \begin{figuretext}\label{peakon.pdf}
      The periodic peakon $\varphi(x)$ of the $\mu$CH equation.
     \end{figuretext}
   \end{center}
\end{figure}

\section{Outline of Proof}\label{outlinesec}
There are two standard methods for studying stability of a solution of a dispersive wave equation. The first method consists of linearizing the equation around the solution. In many cases, nonlinear stability is governed by the linearized equation. However, for the $\mu$CH and CH equations, the nonlinearity plays the dominant role rather than being a higher-order perturbation of the linear terms. Thus, it is not clear how to prove nonlinear stability of the peakons using the linearized problem. Moreover, the peakons $c\varphi(x-ct)$ are continuous but not differentiable, which makes it hard to analyze the spectrum of the operator linearized around $c\varphi$.

The second method is variational in nature. In this approach, the solution is realized as an energy minimizer under appropriate constraints. Stability follows if the uniqueness of the minimizer can be established (otherwise one only obtains the stability of the set of minima). A proof of the stability of the Camassa-Holm peakons using the variational approach is given in \cite{C-L1} for the case on the line and in \cite{le2} for the periodic case.

In this paper, we prove stability of the peakon (\ref{varphipeakon}) using a method that is different from both of the above methods. Taking $c = 1$ for simplicity, our approach can be described as follows.
To each function $w:S^1 \to \R$, we associate a function $F_w(M,m)$ of two real variables $(M, m)$ in such a way that the correspondence $w \mapsto F_w$ has the following properties:

\begin{itemize}
\item If $u(x,t)$ is a solution of (\ref{muCH}) with maximal existence time $T>0$, then
\begin{equation}\label{Fgeqzero}
F_{u(t)}(M_{u(t)},m_{u(t)})\geq 0, \qquad t \in [0,T),
\end{equation}
where $M_{u(t)}= \max_{x \in S^1} \{u(x,t)\}$ and $m_{u(t)}= \min_{x \in {S^1}}\{u(x,t)\}$ denote the maximum and minimum of $u$ at the time $t$, respectively.
\item For the peakon, we have $F_\varphi \equiv F_{\varphi(\cdot)} = F_{\varphi(\cdot - t)}$ and $F_\varphi(M,m) \leq 0$ for all $(M,m)$ with equality if and only if $(M,m) = (M_\varphi, m_\varphi)$, see figure \ref{graphF.pdf}.
\item If $w:S^1 \to \R$ is such that $H_i[w]$ is close to $H_i[\varphi]$, $i = 0,1,2$, where $H_0, H_1, H_2$ are the conservation laws of (\ref{muCH}) given by
\begin{align}  \label{conservationlaws}
H_0[u] = \int u dx, \quad
H_1[u] = \frac{1}{2}\int mu dx, \quad
H_2[u] = \int \left(\mu(u)u^2 + \frac{1}{2}uu_x^2 \right)dx,
\end{align}
then the function $F_{w}$ is a small perturbation of $F_{\varphi}$.
\end{itemize}
Using the correspondence $w \mapsto F_w$, stability of the peakon is proved as follows.
If $u$ is a solution starting close to the peakon $\varphi$, the conserved quantities $H_i[u]$ are close to
$H_i[\varphi]$, $i=0,1,2$, and hence $F_{u(t)}$ is a small perturbation of $F_{\varphi}$ for any $t \in [0, T)$. This implies that the set where $F_{u(t)} \geq 0$ is contained in a small neighborhood of
$(M_\varphi, m_\varphi)$ for any $t \in [0, T)$. We conclude from (\ref{Fgeqzero}) that
$(M_{u(t)},m_{u(t)})$ stays close to
$(M_\varphi,m_\varphi)$ for all times. The proof is completed by noting
that if the maximum of $u$ stays close to the maximum of the peakon,
then the shape of the whole wave remains close to that of the peakon.

Our proof is inspired by \cite{le1} where the stability of the periodic peakons of the Camassa-Holm equation
is proved.\footnote{The proof in \cite{le1} is in turn inspired by the proof of stability of the Camassa-Holm peakons on the line presented in \cite{C-S}.}
 The approach here is similar, but there are differences. The main difference is that in \cite{le1} the function $F_u$ associated with a solution $u(x,t)$ could be chosen to be independent of time, whereas here the function $F_{u(t)}$ depends on time. Indeed, our definition of the function $F_{u(t)}(M,m)$ involves the $L^2$-norm $\|u(t)\|_{L^2({S^1})}$, which is not conserved in time. However, since this norm is controlled by the conservation law $H_1$, we can ensure that it remains bounded for all times. This turns out to be enough to ascertain that the function $F_{u(t)}$, despite its time-dependence, remains close to $F_\varphi$ for all $t \in [0, T)$.

\section{Proof of Stability}\label{proofsec}
We will identify $S^1$ with the interval $[0,1)$ and view
functions on $S^1$ as
periodic functions on the real line of period one. For an integer $n \geq 1$,
we let $H^n(S^1)$ denote the Sobolev space of all square integrable functions
$f \in L^2(S^1)$ with distributional derivatives $\partial_x^i f \in L^2(S^1)$ for
$i=1,\dots,n$. The norm on $H^n(S^1)$ is given by
$$\| f \|_{H^n(S^1)}^2 = \sum_{i=0}^n \int_{S^1} (\partial_x^i f)^2(x) dx.$$
Equation (\ref{muCH}) can be recast in conservation form as
\begin{equation} \label{weakmuCH}
u_t + uu_x
+
A^{-1}\partial_x \Big( 2\mu(u)u +\frac{1}{2}u_x^2 \Big)
= 0,
\end{equation}
where $A= \mu - \partial_x^2$ is an isomorphism between $H^s(S^1)$ and $H^{s-2}(S^1)$ cf. \cite{KLM}.
By a {\it weak solution} $u$ of (\ref{muCH}) on $[0,T)$ with $T>0$, we mean a function $u \in
C([0,T); H^1(S^1))$ such that (\ref{weakmuCH}) holds in distributional sense and
the functionals $H_i[u]$, $i=0,1,2$, defined in (\ref{conservationlaws}) are
independent of $t \in[0,T)$. The peakons defined in (\ref{varphipeakon}) are weak solutions in this sense \cite{lmt}.
Our aim is to prove the following precise reformulation of the theorem
stated in the introduction.

\begin{theorem} \label{thm_stabprecise}
For every $\epsilon >0$ there is a $\delta >0$ such that if
    $u \in C([0,T); H^1({S^1}))$ is a weak solution of (\ref{muCH}) with
    $$ \|u(\cdot, 0) - c\varphi \|_{H^1({S^1})} < \delta$$
  then
    $$ \|u(\cdot, t) - c\varphi (\cdot - \xi (t) + 1/2) \|_{H^1({S^1})} < \epsilon \quad
       \textrm{for} \quad t \in [0,T),$$
  where $\xi(t) \in \R$ is any point where the function $u(\cdot, t)$
  attains its maximum.
\end{theorem}

The proof of theorem \ref{thm_stabprecise} will proceed through a series of lemmas. The first lemma summarizes the properties of the peakon. For simplicity we henceforth take $c=1$.

\begin{lemma}\label{peakonlemma}
The peakon $\varphi(x)$ is continuous on $S^1$ with peak at $x=\pm 1/2$. The extrema of $\varphi$ are
$$ M_{\varphi} = \varphi (1/2) = 1, \qquad
m_{\varphi} = \varphi (0) = \frac{23}{26}.$$
Moreover,
$$\lim_{x \uparrow 1/2} \varphi_x(x) = \frac{6}{13}, \qquad \lim_{x \downarrow -1/2} \varphi_x(x) = -\frac{6}{13},$$
and
\begin{align*}
& H_0[\varphi]
= \frac{12}{13}, \qquad
H_1[\varphi]
= \max_{x \in S^1} \varphi_x = \frac{6}{13}, \qquad
H_2[\varphi]
= \frac{9024}{10985}.
\end{align*}
\end{lemma}
\begin{proof}
These properties follow easily from the definition (\ref{varphipeakon}) of $\varphi$ and the definition (\ref{conservationlaws}) of $\{H_i\}_1^3$. For example,
$$H_0[\varphi] = \int_{-1/2}^{1/2} \frac{12x^2 + 23}{26} dx = \frac{12}{13}.$$
\end{proof}

We define the $\mu$-inner product $\langle \cdot, \cdot \rangle_{\mu}$ and the associated $\mu$-norm $\| \cdot \|_\mu$ by
\begin{equation}\label{muinnernorm}
\langle u, v \rangle_\mu = \mu(u)\mu(v) + \int_{S^1} u_xv_x dx, \qquad \| u \|_\mu^2 = \langle u, u \rangle_\mu = 2H_1[u], \qquad u,v \in H^1(S^1),
\end{equation}
and consider the expansion of the conservation law $H_1$ around the peakon $\varphi$ in the $\mu$-norm. The following lemma shows that the error term in this expansion is given by $12/13$ times the difference between $\varphi$ and the perturbed solution $u$ at the point of the peak.

\begin{lemma}\label{lm_H1est}
For every $u\in H^1({S^1})$ and $\xi \in \R$,
$$ H_1[u]-H_1[\varphi] = \frac{1}{2} \| u- \varphi (\cdot - \xi)\|^2_{\mu}
+ \frac{12}{13}(u(\xi + 1/2) - M_\varphi).$$
\end{lemma}

\begin{proof}
We compute
\begin{align*}
\frac{1}{2} \| u- \varphi (\cdot - \xi)\|^2_{\mu}
&= H_1[u] + H_1[\varphi (\cdot - \xi)] - \mu(u)\mu(\varphi) - \int_{S^1} {u_x(x)\varphi_x(x - \xi)dx}
	\\
& = H_1[u] + H_1[\varphi] - \mu(u) \mu(\varphi) + \int_{S^1} u(x+\xi)\varphi_{xx}(x)dx.
\end{align*}
Since
\begin{equation}\label{eqn_phi_xx}
 \varphi_{xx}= \frac{12}{13} - \frac{12}{13}\delta(x - 1/2),
\end{equation}
we find
$$\int_{S^1} {u(x+\xi)\varphi_{xx}(x)dx}= \frac{12}{13} \int_{S^1} u(x) dx - \frac{12}{13} u(\xi +1/2).$$
Using that $H_0[\varphi]=\mu(\varphi)= \frac{12}{13}$, we obtain
$$\frac{1}{2} \| u- \varphi (\cdot - \xi)\|^2_{\mu}
=H_1[u] - H_1[\varphi] + \frac{12}{13}(1 - u(\xi + 1/2)).$$
This proves the lemma.
\end{proof}

\begin{remark}
For a wave profile $u \in H^1({S^1})$, the functional $H_1[u]$ represents
kinetic energy. Lemma \ref{lm_H1est} implies that if a wave $u \in H^1({S^1})$ has energy $H_1[u]$ and height $M_u$ close to the peakon's energy and height, then the whole shape of $u$ is close to that of the peakon. Another physically relevant consequence of lemma \ref{lm_H1est} is that among all waves of fixed energy, the peakon has maximal height. Indeed, if $u \in H^1({S^1}) \subset C({S^1})$ is such that $H_1[u] = H_1[\varphi]$ and $u(\xi) = \max_{x \in {S^1}} u(x)$, then $u(\xi) \leq M_{\varphi}$.
\end{remark}

The peakon $\varphi$ satisfies the differential equation
\begin{equation} \label{peakondiff}
   \varphi_x = \left\{
   \begin{array}{ll}
     -\frac{12}{13} \sqrt{\frac{13}{6}(\varphi  - m_\varphi)}  & \quad -1/2 < x \leq 0, \\
       \frac{12}{13} \sqrt{\frac{13}{6}(\varphi  - m_\varphi)} & \quad  0 \leq x< 1/2.
   \end{array} \right.
\end{equation}
Let $u \in H^1({S^1}) \subset C({S^1})$ and write $M = M_u =
\max_{x \in {S^1}} \{u(x)\}$,
$m = m_u = \min_{x \in {S^1}}\{u(x)\}$. Let $\xi$ and $\eta$ be such
that $u(\xi)=M$ and $u(\eta)=m$. Inspired by (\ref{peakondiff}),
we define the real-valued function $g(x)$ by
\begin{displaymath}
   g(x) = \left\{
   \begin{array}{ll}
     u_x + \frac{12}{13} \sqrt{\frac{13}{6}(u  - m)} & \quad \xi < x \leq \eta, \\
     u_x - \frac{12}{13}\sqrt{\frac{13}{6}(u  - m)} & \quad \eta \leq x < \xi+1,
   \end{array} \right.
\end{displaymath}
and extend it periodically to the real line. We compute
\begin{align*}
\int_{S^1} {g^2(x) dx}  = &\; \int_{\xi}^{\eta} \left(u_x + \frac{12}{13} \sqrt{\frac{13}{6}(u  - m)}\right)^2 dx
  + \int_{\eta}^{\xi+1} \left(u_x - \frac{12}{13} \sqrt{\frac{13}{6}(u  - m)}\right)^2 dx
  	\\
 =& \; \int_{\xi}^{\eta} {u_x^2 dx} + \frac{24}{13} \int_{\xi}^{\eta} u_x  \sqrt{\frac{13}{6}(u  - m)} dx
+\frac{144}{169} \int_{\xi}^{\eta} \frac{13}{6}(u  - m) dx
	\\
& + \int_{\eta}^{\xi+1} {u_x^2 dx}
- \frac{24}{13}\int_{\eta}^{\xi+1} u_x \sqrt{\frac{13}{6}(u  - m)} dx
  + \frac{144}{169} \int_{\eta}^{\xi+1} \frac{13}{6}(u  - m) dx.
 \end{align*}
Notice that
$$\frac{d}{dx}\left[ 8 \sqrt{\frac{2}{39}} (u - m)^{3/2}\right]
   = \frac{24}{13} u_x \sqrt{\frac{13}{6}(u  - m)} .$$
Hence,
$$\int_{\xi}^{\eta} u_x  \sqrt{\frac{13}{6}(u  - m)}  dx
= -\int_{\eta}^{\xi +1} u_x  \sqrt{\frac{13}{6}(u  - m)} dx$$
and
$$ \frac{24}{13} \int_{\xi}^{\eta}u_x \sqrt{\frac{13}{6}(u  - m)} dx
= \biggl[8 \sqrt{\frac{2}{39}} (u - m)^{3/2}\biggr]_{\xi}^{\eta} = - 8 \sqrt{\frac{2}{39}} (M-m)^{3/2}.$$
We conclude that
\begin{align}\label{gsquare}
   \frac{1}{2} \int_{S^1} {g^2(x) dx}
     =  H_1[u] - \frac{1}{2}\mu(u)^2 - 8 \sqrt{\frac{2}{39}} (M-m)^{3/2}
     + \frac{12}{13}(\mu(u) - m).
\end{align}

In the same way, we compute
\begin{align*}
\int_{S^1} u & g^2(x) dx
	\\
= &\; \int_{\xi}^{\eta} {u\left(u_x + \frac{12}{13}\sqrt{\frac{13}{6}(u  - m)}\right)^2 dx}
+ \int_{\eta}^{\xi+1} {u\left(u_x - \frac{12}{13} \sqrt{\frac{13}{6}(u  - m)} \right)^2 dx}
	\\
=&\; \int_{\xi}^{\eta} {u u_x^2 dx} + \frac{24}{13}\int_{\xi}^{\eta} u u_x \sqrt{\frac{13}{6}(u  - m)} dx
+ \frac{144}{169}  \int_{\xi}^{\eta} u \frac{13}{6}(u  - m) dx
	\\
& + \int_{\eta}^{\xi+1} {u u_x^2 dx} - \frac{24}{13}\int_{\eta}^{\xi+1} {u u_x \sqrt{\frac{13}{6}(u  - m)} dx}
+ \frac{144}{169}  \int_{\eta}^{\xi+1} u \frac{13}{6}(u  - m) dx.
\end{align*}
Since
$$\frac{d}{dx}\left[\frac{8}{5} \sqrt{\frac{2}{39}} (u -m)^{3/2} (2 m+3 u)\right] = \frac{24}{13} u u_x \sqrt{\frac{13}{6}(u  - m)},$$
we find
$$ \int_{\xi}^{\eta} u u_x  \sqrt{\frac{13}{6}(u  - m)} dx
= -\int_{\eta}^{\xi +1} u u_x  \sqrt{\frac{13}{6}(u  - m)} dx$$
and
\begin{align*}
\frac{24}{13} \int_{\xi}^{\eta} u u_x \sqrt{\frac{13}{6}(u  - m)} dx
= - \frac{8}{5} \sqrt{\frac{2}{39}} (M -m)^{3/2} (2 m+3 M).
 \end{align*}
Therefore,
\begin{align}\label{ugsquare}
  \frac{1}{2} \int_{S^1} {u g^2(x) dx} = &\;  H_2[u] - \left(H_0[u] - \frac{12}{13}\right) \int_{S^1} u^2 dx
   - \frac{12}{13} m H_0[u]
   	\\ \nonumber
   &- \frac{8}{5} \sqrt{\frac{2}{39}} (M -m)^{3/2} (2 m+3 M).
\end{align}
Combining (\ref{ugsquare}) with (\ref{gsquare}), we find
\begin{align} \nonumber
   H_2[u] = &\; \frac{1}{2} \int_{S^1} {u g^2(x) dx}
+ \left(H_0[u] - \frac{12}{13}\right) \int_{S^1} u^2 dx
+ \frac{12}{13} m H_0[u]
	\\ \nonumber
& + \frac{8}{5} \sqrt{\frac{2}{39}} (M -m)^{3/2} (2 m+3 M)
      	\\ \label{ineq}
    \leq &\; \frac{M}{2} \int_{S^1} {g^2(x) dx}
   + \left(H_0[u] - \frac{12}{13}\right) \int_{S^1} u^2 dx
+ \frac{12}{13} m H_0[u]
	\\ \nonumber
& + \frac{8}{5} \sqrt{\frac{2}{39}} (M -m)^{3/2} (2 m+3 M)
	\\ \nonumber
  =&\; M \biggl [ H_1[u] - \frac{1}{2}\mu(u)^2 - 8 \sqrt{\frac{2}{39}} (M-m)^{3/2}
     + \frac{12}{13}(\mu(u) - m) \biggr]
     	\\ \nonumber
&   + \left(H_0[u] - \frac{12}{13}\right) \int_{S^1} u^2 dx
+ \frac{12}{13} m H_0[u]
+ \frac{8}{5} \sqrt{\frac{2}{39}} (M -m)^{3/2} (2 m+3 M).
   \end{align}
We have actually proved the following lemma.
\begin{lemma}\label{lm_LyapunovF}
   For any positive $u \in H^1({S^1})$, define a function
$$ {F_u:\{(M,m)\in \R^2:\, M \geq m > 0\} \rightarrow \R} $$
by
\begin{align*}
 F_u(M ,m) = &\; M \biggl [ H_1[u] - \frac{1}{2}H_0[u]^2 - 8 \sqrt{\frac{2}{39}} (M-m)^{3/2}
     + \frac{12}{13}(H_0[u] - m) \biggr]
     	\\ \nonumber
&   + \left(H_0[u] - \frac{12}{13}\right) \int_{S^1} u^2 dx
+ \frac{12}{13} m H_0[u]
	\\ \nonumber
& + \frac{8}{5} \sqrt{\frac{2}{39}} (M -m)^{3/2} (2 m+3 M) - H_2[u].
\end{align*}
Then
$$F_u(M_u ,m_u) \geq 0,$$
where $M_u= \max_{x \in {S^1}} \{u(x)\}$ and $m_u= \min_{x \in {S^1}}\{u(x)\}.$
\end{lemma}

Note that the function $F_{u}$ depends on $u$ only through the three conservation laws $H_0[u]$, $H_1[u]$, and $H_2[u]$, and the $L^2$-norm of $u$.

\begin{figure}
   \begin{center}
     \includegraphics{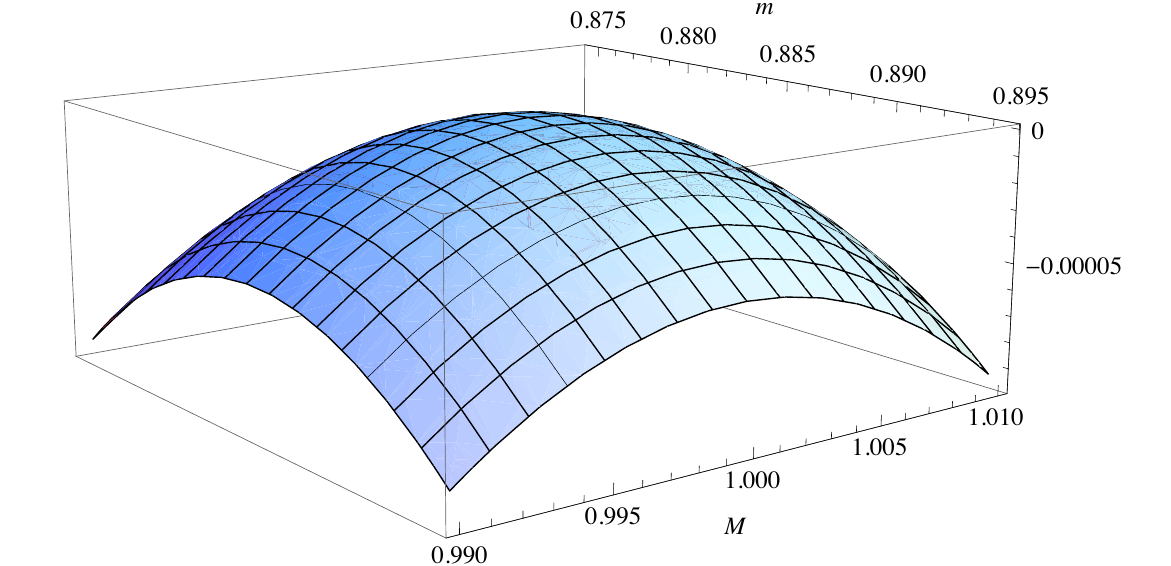} \\
         \begin{figuretext}\label{graphF.pdf}
The graph of the function $F_{\varphi}(M,m)$ near the point $(M_{\varphi},m_{\varphi}).$
     \end{figuretext}
   \end{center}
\end{figure}

The next lemma highlights some properties of the function $ F_{\varphi}(M ,m)$
associated to the peakon. The graph of $F_{\varphi}(M,m)$ is shown in figure \ref{graphF.pdf}.

\begin{lemma}\label{lm_derivativesF}
   For the peakon $\varphi$, we have
   $$ F_{\varphi}(M_{\varphi},m_{\varphi}) = 0,$$
   $$ \frac{\partial F_{\varphi}}{\partial M}(M_{\varphi} ,m_{\varphi}) = 0,
   \qquad \frac{\partial F_{\varphi}}{\partial m}(M_{\varphi}
,m_{\varphi}) = 0,$$
$$ \frac{\partial^2 F_{\varphi}}{\partial M^2}(M_{\varphi} ,m_{\varphi}) = -{12\over13},
\qquad \frac{\partial^2 F_{\varphi}}{\partial M \partial
m}(M_{\varphi} ,m_{\varphi}) = 0,
\qquad \frac{\partial^2 F_{\varphi}}{\partial m^2}(M_{\varphi}
,m_{\varphi}) = -{12\over13}.$$
\end{lemma}

\begin{proof}
It follows from (\ref{peakondiff}) that the function $g(x)$
corresponding
to the peakon is identically zero. Thus the inequality (\ref{ineq}) is an
equality in the case of the peakon. This means that
$ F_{\varphi}(M_{\varphi},m_{\varphi}) = 0.$

On the other hand, differentiation gives
\begin{align*}
\frac{\partial F_u}{\partial M} & = \biggl [ H_1[u] - \frac{1}{2}H_0[u]^2 - 8 \sqrt{\frac{2}{39}} (M-m)^{3/2}
     + \frac{12}{13}(H_0[u] - m) \biggr]
          	\\ \nonumber
	&\quad\   -12\sqrt{{2\over39}}M(M-m)^{1/2}+{12\over5}\sqrt{2\over39}(M-m)^{1/2}(2m+3M)+{24\over5}\sqrt{2\over39}(M-m)^{3/2}\\ \nonumber
	& = \biggl [ H_1[u] - \frac{1}{2}H_0[u]^2 - 8 \sqrt{\frac{2}{39}} (M-m)^{3/2}
     + \frac{12}{13}(H_0[u] - m) \biggr],
\end{align*}
and
\begin{align*}
 \frac{\partial F_u}{\partial m} & = 12\sqrt{{2\over39}}M(M-m)^{1/2} - {12\over13}M+{12\over13}H_0[u]\\
     &\quad\
     + {8\over5}\sqrt{2\over39}\biggl[ -{3\over2}(M-m)^{1/2}(2m+3M) + 2(M-m)^{3/2}  \biggr]\\
     &= {12\over13}(H_0[u]-M)+8\sqrt{2\over39}(M-m)^{3/2}.
\end{align*}
Further differentiation yields
\begin{align*}
    \frac{\partial^2 F_u}{\partial M \partial m}&
     = -{12\over13}+12\sqrt{{2\over39}}(M-m)^{1/2}, \\
    \frac{\partial^2 F_u}{\partial M^2}=\frac{\partial^2 F_u}{\partial m^2}& = -12\sqrt{{2\over39}}(M-m)^{1/2}.
\end{align*}
To complete the proof, take $F_u = F_\varphi$, $M = M_\varphi$,
and $m = m_\varphi$ in the above expressions for the partial derivatives of
$F$ and use lemma \ref{peakonlemma}.
\end{proof}

\begin{lemma}\label{lm_maxf}
We have
   \begin{equation} \label{maxineq}
     \max_{x \in S^1} |f(x)| \leq \sqrt{ \frac{13}{12} }
     \;\|f\|_{\mu}, \quad f \in H^1(S^1),
   \end{equation}
   where the $\mu$-norm is defined in \eqref{muinnernorm}.
 Moreover, $\sqrt{ \frac{13}{12} }$ is the best constant and
 equality holds in (\ref{maxineq}) if and only if $f=c\varphi(\cdot - \xi + 1/2)$ for some
 $c,\xi \in \R$,
 i.e. if and only if $f$ has the shape of a peakon.
\end{lemma}
\begin{proof}
For $x \in S^1$, by \eqref{muinnernorm} and \eqref{eqn_phi_xx}, we have
\begin{align*}
\frac{13}{12} \langle \varphi(\cdot-x+1/2), f \rangle_{\mu}
&= {13\over12}\mu(\varphi(\cdot-x+1/2))\mu(f) + \frac{1}{2} \int_{S^1} {\varphi'(y-x+1/2)f'(y) dy}\\
&= \frac{13}{12} \int_{S^1} {(\mu-\partial^2_y)\varphi(y-x+1/2)f(y) dy}
	\\
&= \int_{S^1} {\delta(y-x)f(y) dy} = f(x)
\end{align*}
Thus, since
$$H_1[\varphi]= \frac{1}{2} \|\varphi\|_{\mu}^2 =
\frac{6}{13},$$
we get
\begin{equation} \label{fineq}
  f(x) = \frac{13}{12} \langle \varphi(\cdot-x+1/2), f \rangle_{\mu}
  \leq \frac{13}{12} \| \varphi \|_{\mu} \|f\|_{\mu}
  = \sqrt{\frac{13}{12}} \; \|f\|_{\mu},
\end{equation}
with equality if and only if $f$ and $\varphi(\cdot-x+1/2)$ are proportional.
Taking the maximum of (\ref{fineq}) over $S^1$ proves the lemma.
\end{proof}
\begin{remark}\label{rk_height}
 Lemma \ref{lm_maxf} again indicates that among all travelling waves of fixed energy, the
peakon has maximal height (see also \cite{C-S, le1}).
\end{remark}

The next lemma shows that the $\mu$-norm is equivalent to the $H^1({S^1})$-norm.

\begin{lemma}\label{lm_equivnorms}
Every $u\in H^1(S^1)$ satisfies
\begin{equation}\label{eqn_equivnorms}
  \|u\|^2_\mu \leq \|u\|^2_{H^1(S^1)} \leq 3\|u\|^2_\mu.
\end{equation}
\end{lemma}
\begin{proof}
The first inequality holds because (by Jensen's inequality)
$$\mu(u)^2 \leq \int_{S^1} u^2 dx, \qquad u \in H^1(S^1).$$

The second inequality holds because, by lemma \ref{lm_maxf},
$$\|u\|^2_{H^1(S^1)}
\leq \max_{x \in S^1} |u(x)|^2 + \int_{S^1} u_x^2 dx
\leq \left(\frac{13}{12} + 1\right) \|u\|_{\mu}^2.$$
\end{proof}

\begin{remark}
The previous two lemmas can also be proved directly using a Fourier series argument. Indeed, for every $f\in H^3(S^1)$ and $\epsilon>0$, we have (cf. the proof of lemma 2 in \cite{C})
\begin{equation}\label{eqn_estmax}
\max_{x\in S^1}f^2(x)\leq {\epsilon+2\over 24}\int_{S^1}f_x^2 dx +{\epsilon+2\over \epsilon}\mu(f)^2.
\end{equation}
The inequality (see lemma 2.6 in \cite{le1})
\begin{equation}\label{maxH1estimate}
  \max_{x \in S^1} |f(x)|^2 \leq \frac{\cosh(1/2)}{2\sinh(1/2)} \|f\|_{H^1(S^1)}^2, \qquad f \in H^1(S^1),
\end{equation}
implies that the map $f \mapsto \max_{x \in S^1} f(x)$ is continuous from $H^1(S^1)$ to $\R$.
Thus, since $H^3$ is dense in $H^1$, equation (\ref{eqn_estmax}) also holds for $f\in H^1(S^1)$.
It follows that, for every $u\in H^1(S^1)$ and every $\epsilon>0$,
\begin{equation}\label{eqn_equivnorms_epsilon}
\|u\|^2_\mu\leq \|u\|^2_{H^1(S^1)}\leq {\epsilon+2\over \epsilon} \mu^2(u) + {\epsilon+26\over24}\int_{S^1}u^2_x dx.
\end{equation}
In particular, we have (taking $\epsilon=1$)
\begin{equation}\label{eqn_equivnorms}
\|u\|^2_\mu\leq \|u\|^2_{H^1(S^1)}\leq 3 \mu^2(u) + {27\over24}\int_{S^1}u^2_x dx\leq 3\|u\|^2_\mu,
\end{equation}
again showing the equivalence of the two norms.

On the other hand, letting $\epsilon = 24$ in (\ref{eqn_estmax}), we recover (\ref{maxineq}). However, the proof we give in lemma \ref{lm_maxf} provides a better idea in concern with the best constant.
\end{remark}

\begin{lemma}\label{lm_contMm}\cite{le1}
   If $u \in C([0,T); H^1({S^1}))$, then
    $$ M_{u(t)}= \max_{x \in {S^1}} u(x,t) \quad \hbox{and} \quad
     m_{u(t)}= \min_{x \in {S^1}} u(x,t)$$
   are continuous functions of $t \in [0,T)$.
\end{lemma}

\begin{lemma}\label{lm_shape}
   Let $u \in C([0,T); H^1({S^1}))$ be a solution of (\ref{muCH}). Given a small
neighborhood
   $\mathcal{U}$ of $(M_\varphi, m_\varphi)$ in $\R^2$, there is a $\delta >0$ such that
   \begin{equation} \label{MminU}
     (M_{u(t)}, m_{u(t)}) \in \mathcal{U} \quad \hbox{for} \quad t \in [0,T) \quad
     if \quad
     \|u(\cdot, 0) - \varphi\|_{H^1({S^1})} < \delta.
   \end{equation}
\end{lemma}

\begin{proof}
Suppose $w \in H^1(S^1)$ is a small perturbation of $\varphi$ such that
$H_i[w]=H_i[\varphi]+\epsilon_i$, $i=0,1,2$. Then
\begin{align*}
F_w(M,m) = F_\varphi(M,m) + M\left [\epsilon_1 - H_0[\varphi]\epsilon_0 - {1\over2}\epsilon_0^2 + \frac{12}{13} \epsilon_0 \right]
 + \epsilon_0\int_{S^1}w^2 dx + {12\over13}m\epsilon_0 - \epsilon_2.
\end{align*}
Suppose $\epsilon_1 < 6/13$ so that $H_1[w] \leq 2H_1[\varphi]$. Then, by lemma \ref{lm_equivnorms},
\begin{equation}\label{L2normbound}
  \int_{S^1}w^2 dx \leq \|w\|_{H^1}^2 \leq 3 \|w\|_\mu^2 = 6 H_1[w] \leq 12 H_1[\varphi] = \frac{72}{13}.
\end{equation}
The point is that $\int_{S^1}w^2 dx$ is bounded. Thus, $F_w$ is a small perturbation of $F_\varphi$. The effect of the
perturbation near the point $(M_\varphi, m_\varphi)$ can be
made arbitrarily small by choosing the $\epsilon_i$'s small.
Lemma \ref{lm_derivativesF} says that $F_\varphi(M_\varphi, m_\varphi) = 0$ and that $F_\varphi$
has a critical point with negative definite second derivative at
$(M_\varphi, m_\varphi)$. By continuity of the second derivative, there is
a neighborhood around $(M_\varphi, m_\varphi)$ where $F_\varphi$ is concave
with curvature bounded away from zero. Therefore, the set where $F_w \geq 0$ near $(M_\varphi, m_\varphi)$ will be contained in a neighborhood of $(M_\varphi, m_\varphi)$.

Now let $\mathcal{U}$ be given as in the statement of the lemma.
Shrinking $\mathcal{U}$ if necessary,
we infer the existence of a $\delta' >0$ such that for $u\in C([0,T); H^1({S^1}))$ with
\begin{equation} \label{HclosetoH}
   |H_i[u] - H_i[\varphi]|<\delta', \qquad i=0,1,2,
\end{equation}
it holds that the set where $F_{u(t)} \geq 0$ near $(M_\varphi, m_\varphi)$
is contained in $\mathcal{U}$ for each $t \in [0,T)$.
By lemma \ref{lm_LyapunovF} and lemma \ref{lm_contMm}, $M_{u(t)}$ and $m_{u(t)}$ are continuous functions of $t \in [0,T)$
and $F_{u(t)}(M_{u(t)},m_{u(t)}) \geq 0$ for $t \in [0,T)$. We conclude that
for $u$ satisfying (\ref{HclosetoH}), we have
$$(M_{u(t)}, m_{u(t)}) \in \mathcal{U} \quad \hbox{for} \quad t \in [0,T) \quad \hbox{if}
\quad (M_{u(0)}, m_{u(0)}) \in \mathcal{U}.$$
However, the continuity of the conserved functionals $H_i:H^1({S^1})\rightarrow \R$,
$i=0,1,2$,
shows that there is a $\delta >0$ such that (\ref{HclosetoH}) holds for
all $u$ with
$$   \|u(\cdot, 0) - \varphi\|_{H^1({S^1})} < \delta. $$
Moreover, in view of the inequality (\ref{maxH1estimate}),
taking a smaller $\delta$ if necessary, we may
also assume that $(M_{u(0)}, m_{u(0)}) \in \mathcal{U}$ if
$\|u(\cdot, 0) - \varphi\|_{H^1({S^1})} < \delta.$
This proves the lemma.
\end{proof}

\noindent {\it Proof of theorem \ref{thm_stabprecise}.}
Let $u \in C([0,T); H^1({S^1}))$ be a solution of (\ref{muCH})
and suppose we are given an $\epsilon >0$. Pick a neighborhood
$\mathcal{U}$ of $(M_\varphi, m_\varphi)$ small enough that
$|M - M_\varphi| < \frac{13\epsilon^2}{144}$ if
$(M,m) \in \mathcal{U}$.
Choose a $\delta > 0 $ as in lemma \ref{lm_shape} so that (\ref{MminU}) holds.
Taking a smaller $\delta$ if necessary we may also assume that
  $$|H_1[u]-H_1[\varphi]|< \frac{\epsilon^2}{12}
  \quad \hbox{if} \quad \|u(\cdot, 0) - \varphi\|_{H^1({S^1})} < \delta.$$
Applying lemma \ref{lm_equivnorms} and lemma \ref{lm_H1est}, we conclude that
\begin{align*}
\|u(\cdot, t) - \varphi (\cdot - \xi (t)) \|_{H^1({S^1})}^2&\leq 3\|u(\cdot, t) - \varphi (\cdot - \xi (t)) \|_\mu^2\\
 & = 6(H_1[u]-H_1[\varphi]) + {72\over13}(M_\varphi - M_{u(t)}) < \epsilon^2,
\qquad t \in [0,T),
\end{align*}
where $\xi(t) \in \R$ is any point where $u(\xi(t) + 1/2, t)=M_{u(t)}.$
This completes the proof of the theorem.$\hfill\Box$ \bigskip

\begin{remark}\label{rk3}
Note that our proof of stability applies to any
  $u \in C([0,T); H^1({S^1}))$ such that $H_i[u]$, $i=0,1,2$, are
  independent of time. The fact that $u$ satisfies (\ref{weakmuCH})
  in distributional sense was actually never used.
\end{remark}

\section{Comments}\label{commentssec}
Some classical solutions of (\ref{muCH}) exist for all time while others develop
into breaking waves \cite{F-L-Q, KLM, lmt}.
If $u_0 \in H^3({S^1})$, then there
exists a maximal time $T=T(u_0)>0$ such that (\ref{muCH}) has a unique solution
$u \in C([0,T); H^3({S^1})) \, \cap \, C^1([0,T); H^2({S^1}))$ with $H_0, H_1, H_2$
conserved. For $u_0 \in H^r({S^1})$ with $r>3/2$, it is known \cite{lmt}
that (\ref{muCH}) has a unique strong solution $u\in C([0,T); H^r({S^1}))$ for some
$T>0$, with $H_0, H_1, H_2$ conserved. However, the peakons do not belong to
the space $H^r({S^1})$ for $r>3/2$. Thus, to describe the peakons one has to study weak
solutions of (\ref{muCH}). The existence and uniqueness of weak solutions to \eqref{muCH} is still open at point.  Therefore, close to a peakon, there may exist profiles that
develop into breaking waves and profiles that lead to globally existing waves.
Our stability theorem is applicable in both cases up to breaking time.




\begin{thebibliography}{99}
\small


\bibitem{C-H}
{\small \textsc{R. Camassa and D. Holm,}\ An integrable shallow water equation with
peaked solitons, {\it Phys. Rev. Lett.} {\bf 71} (1993),
1661--1664.}

\bibitem{C}
{\small \textsc{A. Constantin,}\ On the blow-up of solutions of a periodic shallow water equation,
{\it J. Nonl. Sci.}
{\bf 10} (2000), 391--399.}

\bibitem{con1} {\small \textsc{A. Constantin,}\ On the Cauchy problem for the periodic
Camassa-Holm equation, {\it J. Differential Equations}, {\bf 141}
(1997), 218-235.}

\bibitem{con4} {\small \textsc{A. Constantin, J. Escher,}\ On the blow-up rate and
the blow-up set of breaking waves for a shallow water equation, {\it
Math. Z.}, {\bf 233} (2000), 75-91.}

\bibitem{C-E}
{\small \textsc{A. Constantin and J. Escher,}\ Wave breaking for nonlinear nonlocal
shallow water equations,
{\it Acta Mathematica}
{\bf 181} (1998), 229--243.}




\bibitem{C-M}
{\small \textsc{A. Constantin and H. P. McKean,}\ A shallow water equation on the circle,
{\it Comm. Pure Appl. Math.}
{\bf 52} (1999), 949--982.}


\bibitem{C-L1}
{\small \textsc{A. Constantin and L. Molinet,}\ Orbital stability of solitary waves for a shallow water equation,
{\it Phys. D}
{\bf 157} (2001), 75--89.}

\bibitem{C-S}
{\small \textsc{A. Constantin and W. Strauss,}\ Stability of peakons,
{\it Comm. Pure Appl. Math.}
{\bf 53} (2000), 603--610.}


\bibitem{F-F}
{\small \textsc{A. Fokas, B. Fuchssteiner,}\ Symplectic structures, their B{\"
a}cklund transformation and hereditary symmetries, {\it Physica D}
{\bf 4} (1981), 47--66.}

\bibitem{F-L-Q}
{\small \textsc{Y. Fu, Y. Liu, C. Qu,}\ On the blow-up structure for the generalized periodic Camassa-Holm and Degasperis-Procesi equations, (2010), preprint.}


\bibitem{H-S} {\small \textsc{J. K. Hunter, R. Saxton,}\
Dynamics of director fields, {\it SIAM J. Appl. Math.}, {\bf 51}
(1991), 1498-1521.}

\bibitem{H-Z}{\small \textsc{J. K. Hunter, Y. Zheng,}\
On a completely integrable nonlinear hyperbolic variational equation, {\it Physica D} {\bf 79} (1994), 361--386.}



\bibitem{KLM}
{\small \textsc{B. Khesin, J. Lenells, G. Misio\l ek,}\ Generalized Hunter-Saxton
equation and the geometry of the group of circle diffeomorphisms, {\it Math. Ann.}, {\bf 342}
(2008), 617-656.}

\bibitem{K-M}
{\small \textsc{B. Khesin, G. Misio\l ek,}\
Euler equations on homogeneous spaces and Virasoro orbits, {\it Adv. Math.} {\bf 176} (2003), 116--144.}

\bibitem{kou}{\small \textsc{S. Kouranbaeva,}\ The Camassa-Holm equation as a geodesic
flow on the diffeomorphism group, {\it J. Math. Phys.},  {\bf 40}
(1999), 857-868.}

\bibitem{le1} {\small \textsc{J. Lenells,}\ Stability of periodic peakons, {\it Int. Math. Res. Not.,} {\bf 10}
(2004), 485-499.}

\bibitem{le2} {\small \textsc{J. Lenells,}\ A variational approach to the stability of periodic peakons, {\it J. Nonlinear Math. Phys.,} {\bf 11}
(2004), 151-163.}

\bibitem{lmt}
{\small \textsc{J. Lenells, G. Misio\l ek, F. Ti$\breve{\hbox{g}}$lay}\ Integrable evolution equations on spaces of tensor densities and their peakon solutions, {\it Comm. Math. Phys.}, {\bf 299}
(2010), 129-161.}




\bibitem{mis1} {\small \textsc{G. Misio\l ek,}\ Classical solutions of the periodic Camassa-Holm
equation, {\it Geom. Funct. Anal.}, {\bf 12} (2002), 1080-1104.}

\bibitem{mis2} {\small \textsc{G. Misio\l ek,}\ A shallow water equation as a geodesic
flow on the Bott-Virasoro group, {\it J. Geom. Phys.}, {\bf 24} (1998), 203-208.}




\bibitem{Shk98}
{\small \textsc{S. Shkoller,}\ Geometry and curvature of diffeomorphism groups with $H^1$ metric and mean hydrodynamics, {\it J. Funct. Anal.} {\bf 160} (1998), 337--365.}




\end{thebibliography}
\end{document}